\documentclass[a4paper,12pt]{amsart}
\usepackage[english]{babel}
\usepackage{amsmath,amsthm,amssymb,amsfonts}
\usepackage{multicol}

\usepackage[pdftex]{color}
\usepackage[bookmarks=true,hyperindex,pdftex,colorlinks, citecolor=blue,linkcolor=blue, urlcolor=blue]{hyperref}
\usepackage[dvipsnames]{xcolor}
\usepackage{mathtools}
\usepackage{graphicx}
\usepackage[titletoc]{appendix}
\usepackage{tikz}

\usepackage[normalem]{ulem}

\usetikzlibrary{matrix}

\newtheorem{theorem}{Theorem}[section]
\newtheorem*{theoremA}{Theorem A}
\newtheorem*{theoremB}{Theorem B}

\newtheorem{corollary}[theorem]{Corollary}

\theoremstyle{definition}
\newtheorem{definition}[theorem]{Definition}

\theoremstyle{remark}
\newtheorem{remark}[theorem]{Remark}
\numberwithin{equation}{section}

\newcommand{\R}{\mathbb{R}}

\newcommand{\N}{\mathbb{N}}
\newcommand{\K}{\mathbb{K}}

\DeclareMathOperator{\dist}{dist\,}

\DeclareMathOperator{\re}{Re}

\newcommand{\nn}[1]{{\left\vert\kern-0.25ex\left\vert\kern-0.25ex\left\vert #1 
		\right\vert\kern-0.25ex\right\vert\kern-0.25ex\right\vert}}

\renewcommand{\geq}{\geqslant}
\renewcommand{\leq}{\leqslant}

\newcommand{\NA}{\operatorname{NA}}

\newcommand{\pten}{\ensuremath{\widehat{\otimes}_\pi}}

\newcommand{\eps}{\varepsilon}

\begin{document}
	
	\title{A characterization of a local vector valued Bollob\'as theorem}
	
	\author[Dantas]{Sheldon Dantas}
	\address[Dantas]{Departament de Matem\`{a}tiques and Institut Universitari de Matem\`{a}tiques i Aplicacions de Castell\'o (IMAC), Universitat Jaume I, Campus del Riu Sec. s/n, 12071 Castell\'o, Spain \newline
		\href{http://orcid.org/0000-0001-8117-3760}{ORCID: \texttt{0000-0001-8117-3760} } }
	\email{\texttt{dantas@uji.es}}
 	\urladdr{\url{sheldondantas.com}}

	\author[Rueda Zoca]{Abraham Rueda Zoca}
	\address[Rueda Zoca]{Universidad de Murcia, Departamento de Matem\'aticas, Campus de Espinardo 30100 Murcia, Spain
		\newline
		\href{https://orcid.org/0000-0003-0718-1353}{ORCID: \texttt{0000-0003-0718-1353} }}
	\email{\texttt{abraham.rueda@um.es}}
	\urladdr{\url{https://arzenglish.wordpress.com}}

	\begin{abstract} In this paper, we are interested in giving two characterizations for the so-called property {\bf L}$_{o,o}$, a local vector valued Bollob\'as type theorem. We say that $(X, Y)$ has this property whenever given $\eps > 0$ and an operador $T: X \rightarrow Y$, there is $\eta = \eta(\eps, T)$ such that if $x$ satisfies $\|T(x)\| > 1 - \eta$, then there exists $x_0 \in S_X$ such that $x_0 \approx x$ and $T$ itself attains its norm at $x_0$. This can be seen as a strong (although local) Bollob\'as theorem for operators. We prove that the pair $(X, Y)$ has the {\bf L}$_{o,o}$ for compact operators if and only if so does $(X, \K)$ for linear functionals. This generalizes at once some results due to D. Sain and J. Talponen. Moreover, we present a complete characterization for when $(X \pten Y, \K)$ satisfies the {\bf L}$_{o,o}$ for linear functionals under strict convexity or Kadec-Klee property assumptions in one of the spaces. As a consequence, we generalize some results in the literature related to the strongly subdifferentiability of the projective tensor product and show that $(L_p(\mu) \times L_q(\nu); \K)$ cannot satisfy the {\bf L}$_{o,o}$ for bilinear forms.
	\end{abstract}

	\thanks{ }

	
	\maketitle

	\thispagestyle{plain}

	\section{Introduction}

It has now been 60 years since Bishop and Phelps proved that every bounded linear functional can be approximated by norm-attaining ones \cite{BP}. Since then, several researchers have been working in norm-attaining theory in many different directions and it is out of doubt one of the most traditional topics in Functional Analysis nowadays. Bollob\'as \cite{Bol} pushed further the Bishop-Phelps theorem  by proving that if $\eps > 0$, then there exists $\eta(\eps) > 0$ such that whenever $x^* \in S_{X^*}$ and $x \in S_X$ satisfy $|x^*(x)| > 1 - \eta(\eps)$, then there exist a new functional $x_0^* \in S_{X^*}$ and a new element $x_0 \in S_X$ such that 
\begin{equation} \label{bollobasconditions}
|x_0^*(x_0)| = 1, \ \ \ \ \|x_0 - x\| < \eps, \ \ \ \ \mbox{and} \ \ \ \ \|x_0^* - x^*\| < \eps. 	
\end{equation}

Let us notice the Bishop-Phelps theorem plays an important role in non-reflexive spaces since otherwise {\it every} functional attains its norm. On the other hand, Bollob\'as theorem does make sense in the reflexive setting and, in this case, the functional $x^*$ necessarily attains its norm; so it would be natural to wonder whether a version of Bollob\'as theorem {\it without} changing the initial functional $x^*$ holds in general, that is, whether it is possible to take $x_0^* = x^*$ in (\ref{bollobasconditions}). In a more general situation, we are wondering the following: given $\eps > 0$, is it possible to find $\eta(\eps) > 0$ such that whenever $T \in \mathcal{L}(X, Y)$ with $\|T\| = 1$ and $x \in S_X$ satisfy $\|T(x)\| > 1 - \eta(\eps)$, one can find a new element $x_0 \in S_X$ such that $\|T(x_0)\| = 1$ and $\|x_0 - x\| < \eps$? It is easy to see that the pair $(X, \K)$ satisfies it whenever $X$ is a uniformly convex Banach space and it turns out that this is in fact a characterization for uniformly convex spaces (see \cite[Theorem 2.1]{KL}). Nevertheless, there is no way of getting such a similar statement for linear operators: indeed, the authors in \cite{DKKLM} proved that if $X$ and $Y$ are real Banach spaces of dimension greater than or equal to 2, then the pair $(X, Y)$ always fails such a property. Therefore, {\it the only} hope for getting positive results in the context of operators would be by considering a weakening of the mentioned property and that was done in \cite{D, DKLM, DKLM0, Deb, T} (and more recently in \cite{DGJR, DJRR} as a tool to get positive results on different norm-attainment notions). More specifically, we have the following property.

\begin{definition} Let $X, Y$ be Banach spaces. We say that the pair $(X, Y)$ has the {\bf L}$_{o,o}$ for operators if given $\eps > 0$ and $T \in \mathcal{L}(X, Y)$ with $\|T\| = 1$, there exists $\eta(\eps, T) > 0$ such that whenever $x \in S_X$ satisfies $\|T(x)\| > 1 - \eta(\eps, T)$, 
there exists $x_0 \in S_X$ such that
\begin{equation*}
	\|T(x_0)\| = 1 \ \ \ \ \mbox{and} \ \ \ \ \|x_0 - x\| < \eps. 
\end{equation*}
\end{definition}
Notice that if the pair $(X, Y)$ satisfies such a property, we are saying that {\it every} operator has to be norm-attaining and, consequently, the Banach space $X$ must be reflexive by the James theorem. By using a result due to G. Godefroy, V. Montesinos, and V. Zizler \cite{GMZ} and a characterization by C. Franchetti and R. Pay\'a \cite{FP}, it turns out that the pair $(X, \K)$ has the {\bf L}$_{o,o}$ for linear functionals if and only if $X^*$ is strongly subdifferentiable (SSD, for short; see its definition below). On the other hand, at the same way that it happens in the classical norm-attainment theory (see, for instance, \cite{martinjfa}), the {\bf L}$_{o,o}$ was studied for compact operators \cite{Deb, T}.  It is known that whenever $X$ is strictly convex, the {\bf L}$_{o,o}$ for compact operators is equivalent to saying that the dual $X^*$ is Fr\'echet differentiable (see \cite[Theorem 2.3]{T}); and when $X$ satisfies the Kadec-Klee property then $(X, Y)$ has the {\bf L}$_{o,o}$ for compact operators for every Banach space $Y$ (see \cite[Theorem 2.12]{Deb}).

Our first aim in the present paper is to generalize \cite[Theorem 2.3]{T} and \cite[Theorem 2.12]{Deb} at once. Indeed, we have the following theorem.

\begin{theoremA} \label{characterization:compactos} Let $X$ be a reflexive Banach space. The following are equivalent.
\begin{enumerate}
\item[(i)] The pair $(X, Y)$ satisfies the {\bf L}$_{o,o}$ for compact operators for every Banach space $Y$.  
\item[(ii)] The pair $(X,\mathbb K)$ has the {\bf L}$_{o,o}$ (equivalently, $X^*$ is SSD).
\end{enumerate}
\end{theoremA}

Our second main result deals with a strengthening of the {\bf L}$_{o,o}$ in the context of bilinear forms (see \cite{DKLM}).

\begin{definition}\cite[Definition 2.1]{DKLM} Let $X, Y$ be Banach spaces. We say that $(X \times Y; \K)$ has the {\bf L}$_{o,o}$ for bilinear forms if given $\eps > 0$ and $B \in \mathcal{B}(X \times Y; \K)$ with $\|B\| = 1$, there exists $\eta (\eps, B) > 0$ such that whenever $(x, y) \in S_X \times S_Y$ satisfies $|B(x, y)| > 1 - \eta(\eps, B)$, there exists $(x_0, y_0) \in S_X \times S_Y$ such that
	\begin{equation*}
		|B(x_0, y_0)| = 1, \ \ \ \ \|x_0 - x\| < \eps, \ \ \ \ \mbox{and} \ \ \ \ \|y_0 - y\| < \eps.
	\end{equation*}
\end{definition}
\noindent
It is known that $(X \times Y; \K)$ satisfies the {\bf L}$_{o,o}$ for bilinear forms whenever 
\begin{itemize}
	\item[(a)] $X, Y$ are finite dimensional;
	\item[(b)] $X$ finite dimensional and $Y$ uniformly convex; 
	\item[(c)] $X = \ell_p$ and $Y = \ell_q$ if and only if $p > q'$, where $q'$ is the conjugate index of $q$.
\end{itemize}
(Proposition 2.2.(a), Lemma 2.6, and Theorem 2.7.(b) of \cite{DKLM}, respectively). By observing items (a), (b), and (c) above, one might think that the reflexivity of $X \pten Y$ plays an important role here (notice that (c) gives the result for $\ell_p$-spaces {\it exactly} when the projective tensor product $\ell_p \pten \ell_q$ is reflexive (see \cite[Corollary 4.24]{rya})). And this is indeed not a coincidence: we have the following result, which gives a complete characterization for the {\bf L}$_{o,o}$ in terms of the reflexivity of $X \pten Y$ and also relates the {\bf L}$_{o,o}$ in different classes of functions under strict convexity or Kadec-Klee property assumptions on $X$.

\begin{theoremB} \label{characterization:tensor} Let $X$ be a strictly convex Banach space or a Banach space satisfying the Kadec-Klee property. Let $Y$ be an arbitrary Banach space and assume that either $X$ or $Y$ enjoys the approximation property. The following are equivalent.
\begin{itemize}
	\item[(i)] $(X \pten Y, \K)$ has the {\bf L}$_{o,o}$ for linear functionals.  
	\item[(ii)] $X \pten Y$ is reflexive and both $(X, \K), (Y, \K)$ have the {\bf L}$_{o,o}$ for linear functionals. 
	\item[(iii)] $(X \times Y; \K)$ has the {\bf L}$_{o,o}$ for bilinear forms. 
\end{itemize}
\end{theoremB}

As a consequence of Theorem B, we have that $(L_p(\mu) \times L_q(\nu); \K)$ {\it cannot} satisfy the {\bf L}$_{o,o}$ for bilinear forms for every $1 < p, q < \infty$ since $L_p (\mu) \pten L_q (\nu)$ is never reflexive (see \cite[Theorem 4.21 and Corollary 4.22]{rya}). We conclude the paper with a discussion about the relation between the different properties {\bf L}$_{o,o}$ in $\mathcal B(X\times Y,\mathbb K)$.

\subsection{Terminology and Background} Here will be working with Banach spaces over the real or complex field $\K$. The unit ball and unit sphere of a Banach space $X$ are denoted by $B_X$ and $S_X$, respectively. The symbols $\mathcal{L}(X, Y)$ and $\mathcal{B}(X \times Y; \K)$ stand for the (bounded) linear  operators and bilinear forms, respectively. When $Y =\K$, $\mathcal{L}(X, Y)$ becomes simply $X^*$, the topological dual space of $X$. We say that $T \in \mathcal{L}(X, Y)$ attains its norm if $\|T(x_0)\| = \|T\|$ for some $x_0 \in S_X$ and we say that $B \in \mathcal{B}(X \times Y; \K)$ attains its norm if $|B(x_0, y_0)| = \|B\|$ for some $(x_0, y_0) \in S_X \times S_Y$.

The norm of $X$ is said to be strongly subdifferentiable (SSD, for short) at the point $x\in X$ if the one-side limit 
\begin{equation*} 
\lim_{t \rightarrow 0^+} \frac{\|x + th\| - \|x\|}{t} 
\end{equation*} 
exists uniformly for $h \in B_X$. Let us notice that the norm of $X$ is Fr\'echet differentiable at $x$ if and only it is Gâteux differentiable and SSD at $x$. When we say that $X$ is SSD we mean that the norm of $X$ is SSD at every $x \in S_X$.

The projective tensor product of two Banach spaces $X$ and $Y$ is the completion of $X \otimes Y$ endowed with the norm given by 
\begin{equation*}
\|z\|_{\pi} = \inf \left\{ \sum_{n=1}^{\infty} \|x_n\| \|y_n\|: \sum_{n=1}^{\infty} \|x_n\| \|y_n\| < \infty, z = \sum_{n=1}^{\infty} x_n \otimes y_n \right\}. 
\end{equation*}
We denote the projective tensor product of $X$ and $Y$ endowed with the above norm by $X \pten Y$. It is well-known (and we will be using these facts with no explicit mention throughout the paper) that $\|x \otimes y\| = \|x\| \|y\|$ for every $x \in X$ and $y \in Y$, and that the closed unit ball of $X \pten Y$ is the closed convex hull of the $B_X \otimes B_Y = \{x \otimes y: x  \in B_X, y \in B_Y\}$.	Moreover, we have that $(X \pten Y)^* = \mathcal{B}(X \times Y; \K)$ under the action of a bounded bilinear form $B$ as a bounded linear functional on $X \pten Y$ given by
\begin{equation*}
	\left\langle B, \sum_{n=1}^{\infty} x_n \otimes y_n \right\rangle = \sum_{n=1}^{\infty} B(x_n, y_n)
\end{equation*}
and $(X \pten Y)^* = \mathcal{L}(X, Y^*)$ under the action of a bounded linear operator $T$ as a bounded linear functional on $X \pten Y$ given by
\begin{equation*}
	\left\langle T,  \sum_{n=1}^{\infty} x_n \otimes y_n \right \rangle = \sum_{n=1}^{\infty} \langle y_n, T(x_n) \rangle.
\end{equation*}	
Analogously, we have that $(X \pten Y)^* = \mathcal{L}(Y, X^*)$. 

Recall that a Banach space is said to have the approximation property (AP, in short) if for every compact subset $K$ of $X$ and every $\eps > 0$, there exists a finite-rank operator $T: X \longrightarrow X$ such that $\|T(x) - x\| \leq \eps$ for every $x \in K$. We refer the reader to \cite{rya} for background on the beautiful tensor products of Banach spaces and approximation properties theories.

Finally, let us recall that a Banach space $X$ satisfies the Kadec-Klee property if weak and norm topologies coincide in the unit sphere of $X$.

\section{Proofs of Theorems A and B} 	
	
We start this section by giving the proof of Theorem A.

\begin{proof}[Proof of Theorem A] (i)$\Rightarrow$(ii). Suppose that $(X, Y)$ has the {\bf L}$_{o,o}$ for compact operators.  Let $\eps > 0$ and $x^* \in S_{X^*}$ be given, and let us prove that $(X, \K)$ has the {\bf L}$_{o,o}$ for linear functionals. Define $T: X \longrightarrow Y$ by $T(x) := x^*(x) y_0$ for some $y_0 \in S_Y$. Then, $\|T\| = \|x^*\| = 1$ and $T$ is compact. By hypothesis, there is $\eta(\eps, T) > 0$ witnessing the definition of the property {\bf L}$_{o,o}$. Let us set $\eta(\eps, x^*) := \eta(\eps, T) > 0$. Let $x_0 \in S_X$ be such that $|x^*(x_0)| > 1 - \eta(\eps, T)$. Then, $\|T(x_0)\| = \|x^*(x_0) y_0\| = |x^*(x_0)| > 1 - \eta(\eps, T)$ and by the assumption there is $x_1 \in S_X$ such that $\|T(x_1)\| = 1$ and $\|x_1 - x_0\| < \eps$. Then, $|x^*(x_1)| = \|T(x_1)\| = 1$ and $\|x_1 - x_0\| < \eps$, that is, $(X, \K)$ has the {\bf L}$_{o,o}$ for linear functionals.

\vspace{0.2cm}
\noindent
(ii) $\Rightarrow$ (i). Suppose that $(X, \K)$ has the {\bf L}$_{o,o}$ for linear functionals. By contradiction, suppose that there exist $\eps_0 > 0$, $T \in \mathcal{K}(X, Y)$ with $\|T\| = 1$, and $(x_n) \subseteq S_X$ such that
\begin{equation} \label{ineq1}
	1 \geq \|T(x_n)\| \geq 1 - \frac{1}{n}
\end{equation}
but satisfying that $\dist(x_n, \NA(T)) \geq \eps_0$, where $\NA(T) = \{x \in S_X: \|T(x)\| = \|T\|\}$. Since $X$ is reflexive and $(x_n)_{n=1}^{\infty}$ is bounded, we may (and we do) assume that $x_n \stackrel{w}{\longrightarrow} x_0$ for some $x_0 \in B_X$. Since $T$ is a compact operator, we have that $T(x_n) \stackrel{\|\cdot\|}{\longrightarrow} T(x_0)$. By (\ref{ineq1}), we have that $\|T(x_0)\| = 1$ and, in particular, $x_0 \in S_X$. Let us take $y_0^* \in S_{Y^*}$ to be such that $y^*(T(x_0)) = \|T(x_0)\| = 1$. Consider $x_0^* := T^* y_0^* \in S_{X^*}$. Then $x_0^*(x_0) = T^*y_0^*(x_0) = y^*(T(x_0)) =  1$. Since $x_n \stackrel{w}{\longrightarrow} x_0$, we have that $x_0^*(x_n) \longrightarrow x_0^*(x_0) = 1$ as $n \rightarrow \infty$. Since $(X, \K)$ has the {\bf L}$_{o,o}$ for linear functionals, there is $(x_n') \subseteq S_X$ such that $x_0^*(x_n') = 1$ and $\|x_n' - x_n\| \rightarrow 0$ for every $n \in \N$. This shows that 
\begin{equation*}
	1 = x_0^*(x_n') = T^* y_0^*(x_n') = y_0^*(T(x_n')),
\end{equation*}	
that is, $1 = y_0^*(T(x_n')) \leq \|T(x_n')\| \leq \|T\| = 1$, that is, $\|T(x_n')\| = 1$ and then $x_n' \in \NA(B)$. The convergence $\|x_n' - x_n\| \rightarrow 0$ yields the desired contradiction.
\end{proof}

\begin{remark} We have that \cite[Theorem 2.3]{T} says that $X$ is strictly convex and the pair $(X, Y)$ has the {\bf L}$_{o,o}$ for compact operators if and only if $X^*$ is Fr\'echet differentiable. Let us notice that $X^*$ is Fr\'echet differentiable if and only $X$ is strictly convex and $(X, \K)$ has the  {\bf L}$_{o,o}$ (see \cite[Theorem 2.5]{DKLM0}). Therefore, Theorem \ref{characterization:compactos} generalizes \cite[Theorem 2.3]{T} as we no longer need strict convexity on $X$. On the other hand, \cite[Theorem 2.12]{Deb} says that if $X$ is a reflexive space which satisfies the Kadec-Klee property, then $(X, Y)$ has the {\bf L}$_{o,o}$ for compact operators for every $Y$. This is also covered by our Theorem \ref{characterization:compactos} since whenever $X$ is a reflexive space satisfying the Kadec-Klee property, the pair $(X, \K)$ satisfies the {\bf L}$_{o,o}$ (see \cite[Propositions 2.2 and 2.6]{DKLM0}). 
\end{remark}

We now present the proof of Theorem B.

\begin{proof}[Proof of Theorem B] (i) $\Rightarrow$ (ii). Suppose that $(X \pten Y, \K)$ has the {\bf L}$_{o,o}$ for linear functionals. By \cite[Theorem 2.3]{DKLM0}, $X \pten Y$ is reflexive and $(X \pten Y)^*$ is SSD. Since $X^*, Y^*$ are closed subspaces of $(X \pten Y)^* = \mathcal{L}(X, Y^*) = \mathcal{L}(Y^*, X)$, we have that both $X^*, Y^*$ are SSD \cite{FP}. Therefore, $(X, \K)$ and $(Y, \K)$ satisfy the {\bf L}$_{o,o}$ for linear functionals.

For the proofs of (ii) $\Rightarrow$ (iii) and (iii) $\Rightarrow$ (i), we assume that both $X, Y$ are real Banach spaces. We invite the reader to go to Remark \ref{complexcase} below for some comments on the complex case.

\vspace{0.2cm}
\noindent
(ii) $\Rightarrow$ (iii). Suppose that $X \pten Y$ is reflexive and assume that both $(X, \R)$ and $(Y, \R)$ satisfy the {\bf L}$_{o,o}$ for linear functionals. By contradiction, let us assume that $(X \times Y; \R)$ fails to have the {\bf L}$_{o,o}$ for bilinear forms. Then, there exist $\eps_0 > 0$, $B \in \mathcal{B}(X \times Y; \R)$ with $\|B\| = 1$, and $(x_n, y_n)_{n=1}^{\infty} \subseteq S_X \times S_Y$ such that
\begin{equation} \label{ineq2}
	1 \geq B(x_n, y_n) \geq 1 - \frac{1}{n}
\end{equation}	
and whenever $(x_n', y_n') \subset S_X \times S_Y$ is such that $B(x_n', y_n') = 1$, we have that 
\begin{equation} \label{ineq3} 
\|x_n' - x_n\| \geq \eps_0 \ \ \ \ \mbox{and} \ \ \ \ \|y_n' - y_n\| \geq \eps_0.
\end{equation}

Since $X$ and $Y$ are reflexive and both $(x_n)_{n=1}^{\infty}$ and $(y_n)_{n=1}^{\infty}$ are bounded, we may assume (and we do) that $x_n \stackrel{w}{\longrightarrow} x_0$ and $y_n \stackrel{w}{\longrightarrow} y_0$ for some $x_0 \in B_X$ and $y_0 \in B_Y$.  Now, since $X \pten Y$ is reflexive and we are assuming that either $X$ or $Y$ enjoys the AP, by \cite[Theorem 4.21]{rya}, we have that $(X \pten Y)^* = \mathcal{L}(X, Y^*) = \mathcal{K}(X, Y^*)$. Let $T \in \mathcal{K}(X, Y^*)$ be arbitrary. Since $x_n \stackrel{w}{\longrightarrow} x_0$ and $T$ is completely continuous (i.e. $T(K)$ is a compact  subset of $Y^*$ whenever $K$ is a weakly compact subset of $X$), we have that $T(x_n) \stackrel{\|\cdot\|}{\longrightarrow} T(x_0)$ and then since
\begin{equation*}
	|T(x_n)(y_n) - T(x_0)(y_0)| \leq \|T(x_n) - T(x_0)\| \|y_n\| + |T(x_0)(y_n) - T(x_0)(y_0)|,
\end{equation*}	
we have that 
\begin{equation} \label{conv1} 
	T(x_n)(y_n) \longrightarrow T(x_0)(y_0)
\end{equation}	
as $n \rightarrow \infty$ for every $T \in \mathcal{K}(X, Y^*) = (X \pten Y)^*$. This means that $x_n \otimes y_n \stackrel{w}{\longrightarrow} x_0 \otimes y_0$. In particular, since $B \in \mathcal{B}(X \times Y; \R) = (X \pten Y)^*$, we have that $B(x_n, y_n) \longrightarrow B(x_0, y_0)$ and by (\ref{ineq2}), $B(x_0, y_0) = 1$. In particular, $x_0 \in S_X$ and $y_0 \in S_Y$.

Let us consider $T_B \in \mathcal{L}(X, Y^*)$ and $S_B \in \mathcal{L}(Y, X^*)$ to be the associated linear operators to the bilinear form $B$. Therefore, we have that 
\begin{equation*}
	T_B^*(y_0)(x_0) = T_B(x_0)(y_0) = B(x_0, y_0) = 1,
\end{equation*}
which shows that $T_B^*(y_0) \in S_{X^*}$. Analogously, we have that $S_B^*(x_0) \in S_{Y^*}$.

\vspace{0.2cm} 
\noindent
{\it Claim}: We have that
\begin{itemize}
	\item[(a)] $T_B^*(y_0)(x_n) \longrightarrow 1$ as $n \rightarrow \infty$. 
	\item[(b)] $S_B^*(x_0)(y_n) \longrightarrow 1$ as $n \rightarrow \infty$. 
\end{itemize}

We prove (a) since (b) is analogous. As $T_B^*$ is a compact operator and $y_n \stackrel{w}{\longrightarrow} y_0$, we have that $T_B^*(y_n) \stackrel{\|\cdot\|}{\longrightarrow} T_B^*(y_0)$. At the same time, by (\ref{conv1}) we have that 
\begin{equation*}
	T_B^*(y_n)(x_n) = T_B(x_n)(y_n) \longrightarrow T_B(x_0)(y_0) = 1
\end{equation*}
as $n \rightarrow \infty$. Therefore, 
\begin{eqnarray*}
	|T_B^*(y_n)(x_n) - T_B^*(y_0)(x_n)| &=& |(T_B^*(y_n) - T_B^*(y_0)(x_n))| \\
	&\leq& \|T_B^*(y_n) - T_B^*(y_0)\| \longrightarrow 0
\end{eqnarray*}
and then $T_B^*(y_0)(x_n) \longrightarrow 1$ as $n \rightarrow \infty$.

\vspace{0.3cm}

Now we prove that $\|x_n - x\| \longrightarrow \infty$ as $n \rightarrow \infty$. Assume first that $X$ satisfies the Kadec-Klee property. Since $x_0 \in S_X$ and $x_n \stackrel{w}{\longrightarrow} x_0$, we have that $\|x_n - x_0\| \rightarrow 0$ as $n \rightarrow \infty$. We prove that the same holds if $X$ is taken to be strictly convex. Indeed, by using item (a) of Claim, we have that $T_B^*(y_0)(x_n) \longrightarrow 1$ as $n \rightarrow \infty$. Since $(X, \R)$ satisfies the {\bf L}$_{o,o}$ and $X$ is strictly convex, we have that $X^*$ is Fr\'echet differentiable (see \cite[Theorem 2.5.(b)]{DKLM0} and then, by the \v{S}mulyan lemma, we have that $\|x_n - x_0\| \longrightarrow 0$ as $n \rightarrow \infty$ as desired.

To conclude the proof of this implication, let us set $y_0^* := S_B^*(x_0) \in S_{Y^*}$. Then, $y_0^*(y_0) = 1$ and $y_0^*(y_n) \longrightarrow 1$ as $n \rightarrow \infty$ by Claim (b). Since $(Y, \R)$ has the {\bf L}$_{o,o}$ for linear functionals, there is $(y_n') \subseteq S_Y$ such that $y_0^*(y_n') = 1$ and $\|y_n' - y_n\| \longrightarrow 0$ as $n \rightarrow \infty$. This means that 
\begin{equation*}
1 = y_0^*(y_n') = S_B^*(x_0)(y_n') = B(x_0, y_n').	
\end{equation*} 
Since $\|x_n - x_0\| \longrightarrow 0$ and $\|y_n' - y_n\| \longrightarrow 0$ as $n \rightarrow \infty$, we get a contradiction with (\ref{ineq3}).

\vspace{0.2cm}
\noindent
(iii) $\Rightarrow$ (i). Suppose that $(X \times Y; \R)$ has the {\bf L}$_{o,o}$ for bilinear forms. To prove that  $(X \pten Y, \R)$ has the {\bf L}$_{o,o}$ for linear functionals, let us fix $B \in (X \pten Y)^* = \mathcal{B}(X \times Y; \R)$ with $\|B\| = 1$. Let us take $\eps>0$ to be such that $2 \eps (1 + \eps^2) + \eps < 1$. By hypothesis, there is $\eta(\eps, B) > 0$.

Let $z \in S_{X \pten Y}$ be such that $z= \sum_{n=1}^{\infty} \lambda_n x_n \otimes y_n$, where $(x_n)_{n=1}^{\infty} \subseteq S_X$, $(y_n)_{n=1}^{\infty} \subseteq S_Y$, $(\lambda_n)_{n=1}^{\infty} \subseteq (0,1)$ with 
\begin{equation*}
\sum_{n=1}^{\infty} \lambda_n < 1 + \frac{\eps^2}{2}
\end{equation*}
satisfying $\langle B, z \rangle > 1 - \widetilde{\eta}(\eps, B)$, where
\begin{equation*}
\widetilde{\eta} (\eps, B) := \min \left\{ \frac{\eps}{2}, \eta(\eps, B) \right\} > 0.
\end{equation*}
Consider the sets
\begin{equation*}
I:= \Big\{ n \in \N: B(x_n, y_n) > 1 - \min \{ \eps, \eta(\eps, B) \Big\} 
\end{equation*}
and $J:= I^c$. Then,
\begin{eqnarray*}
1 - \widetilde{\eta}(\eps, B) < \langle B, z \rangle &=& \sum_{n=1}^{\infty} \lambda_n B(x_n, y_n) \\
&\leq& \sum_{n \in I} \lambda_n + (1 - \eps) \sum_{n \in J} \lambda_n \\
&=& \sum_{n=1}^{\infty} \lambda_n - \eps \sum_{n \in J} \lambda_n \\ 
&<& 1 + \frac{\eps^2}{2} - \eps \sum_{n \in J} \lambda_n.	
\end{eqnarray*} 
This clearly implies that
\begin{equation} \label{sum-ineq} 
\sum_{n \in J} \lambda_n < \eps.
\end{equation}
On the other hand, for every $n \in I$, we have that $B(x_n, y_n) > 1 - \eta(\eps, B)$. Since $(X \times Y; \R)$ has the {\bf L}$_{o,o}$ for bilinear forms, there is $(x_n', y_n') \subseteq S_X \times S_Y$ such that 
\begin{equation*}
B(x_n', y_n') = 1, \ \ \ \ \|x_n' - x_n\| < \eps, \ \ \ \ \mbox{and} \ \ \ \ \|y_n' - y_n\| < \eps 
\end{equation*}
for every $n \in I$. Let us then define the tensor
\begin{equation*}
z' := \sum_{n \in I} \lambda_n x_n' \otimes y_n' \ \in X \pten Y.
\end{equation*}
So, we have that 
\begin{equation*}
\langle B, z' \rangle = \sum_{n \in I} \lambda_n B(x_n', y_n') = \sum_{n \in I} \lambda_n = \|z'\|_{\pi}.
\end{equation*}
On the other hand, since 
\begin{eqnarray*}
\|x_n' \otimes y_n' - x_n \otimes y_n\| &\leq& \|x_n' \otimes y_n' - x_n' \otimes y_n\| + \|x_n' \otimes y_n - x_n \otimes y_n\| \\
&\leq& \|y_n' - y_n\| + \|x_n' - x_n\| \\
&<& 2 \eps
\end{eqnarray*} 
we have, by using (\ref{sum-ineq}), that
\begin{eqnarray*}
\|z' - z\|_{\pi} &=& \left\| \sum_{n \in I} \lambda_n x_n' \otimes y_n' - \sum_{n \in I} \lambda_n x_n \otimes y_n - \sum_{n \in J} \lambda_n x_n \otimes y_n \right\| \\
&\leq& \sum_{n \in I} \lambda_n \|x_n' \otimes y_n' - x_n \otimes y_n \| + \sum_{n \in J} \lambda_n \\
&<& 2 \eps \sum_{n \in I} \lambda_n + \sum_{n \in J} \lambda_n \\
&<& 2\eps (1 + \eps^2) + \eps. 	
\end{eqnarray*} 
So, $\|z'\| > \|z\| - 2\eps(1 + \eps^2) + \eps = 1 - 2\eps(1 + \eps^2) + \eps>0$. Finally, let us set $z'' := \frac{z'}{\|z'\|} \in S_{X \pten Y}$. We have that 
\begin{equation*}
\|z'' - z'\| = \left\| \frac{z'}{\|z'\|} - z' \right\| = |1 - \|z'\|| \leq \|z - z'\| < 2\eps(1 + \eps^2) + \eps
\end{equation*}
and so 
\begin{equation*}
\|z'' - z\| \leq \|z'' - z'\| + \|z' - z\| < 4 \eps (1 + \eps^2) + 2 \eps. 
\end{equation*}
Moreover, 
\begin{equation*} 
\langle B, z'' \rangle = \left\langle B, \frac{z'}{\|z'\|} \right\rangle = 1.
\end{equation*} 
This shows that $(X \pten Y, \R)$ satisfies the {\bf L}$_{o,o}$ for linear functionals. 
\end{proof}

\begin{remark} \label{complexcase} Let us comment on the proof of Theorem B in the complex case. It is clear that (i) $\Rightarrow$ (ii) works in both cases immediately. In (ii) $\Rightarrow$ (iii), one can get (\ref{ineq2}) by multiplying $B(x_n, y_n)$ for suitable rotations $e^{i \theta_n}$ for $\theta_n \in \R$ and then the proof follows the same lines. The most delicate implication is (iii) $\Rightarrow$ (ii). In this implication, what we do is to define
	\begin{equation*}
		I:= \Big\{ n \in \N: \re B(x_n, y_n) > 1 - \min \{ \eps, \eta(\eps, B) \} \Big\}. 
	\end{equation*} 
When $n \in I$, we use the assumption that $(X \times Y; \K)$ has the {\bf L}$_{o,o}$ for bilinear forms, to get $(x_n', y_n') \subseteq S_X \times S_Y$ such that 
\begin{equation*} 
|B(x_n', y_n')| = 1, \ \ \ \ \|x_n' - x_n\| < \eps, \ \ \ \ \mbox{and} \ \ \ \  \|y_n' - y_n\| < \eps. 
\end{equation*} 
We then write $B(x_n', y_n') = e^{i \theta_n}$ with some $\theta_n \in \R$ for every $n \in I$ and use the fact that $|1 - e^{i \theta}| < \sqrt{2 \eta(\eps, B)}$ to get that
\begin{equation*}
	z' := \sum_{n \in I} \lambda_n e^{-i \theta_n} x_n' \otimes y_n' 
\end{equation*}
 is arbitrarily close to $z$ and $\left\langle B, \frac{z'}{\|z'\|} \right\rangle = 1$.
\end{remark}

\begin{remark}\label{remark:ap}
An inspection in the proof of Theorem B reveals that the assumption that $X$ or $Y$ enjoys the AP is just used in the implication (ii)$\Rightarrow$(iii); the other implications hold in complete generality.
\end{remark}

As an immediate consequence of Theorem B, we have the following corollary. Notice that item (a) below was proved also in \cite[Theorem 2.7.(b)]{DKLM}.

\begin{corollary} Let $1 < p, q < \infty$ and let $q'$ be the conjugate index of $q$. 
		\begin{itemize}
		\item[(a)] $(\ell_p \times \ell_q; \K)$ satisfies the {\bf L}$_{o,o}$ for bilinear forms if and only if $p > q'$.
		\item[(b)] $(L_p(\mu), L_q(\nu); \K)$ fails the {\bf L}$_{o,o}$ for bilinear forms for non-atomic measures $\mu, \nu$.
	\end{itemize}
\end{corollary}
	
\begin{proof} Under the assumption of (a), we have that the projective tensor product $\ell_p \pten \ell_q$ is reflexive (see \cite[Corollary 4.24]{rya}). For (b), since $L_p (\mu) \pten L_q (\nu)$ contains complemented isomorphic copies of $\ell_1$ for every $p, q$, it is never reflexive (see \cite[Theorem 4.21 and Corollary 4.22]{rya}). Therefore, both items follow immediately by applying Theorem B.
\end{proof}

Let us conclude the paper by commenting on the {\bf L}$_{o,o}$ for different classes of functions. Let $X$ and $Y$ be Banach spaces. In $\mathcal{B}(X \times Y; \K)$, as we have seen in Theorem B, one can consider:
\begin{itemize}
	\item[(A)] the {\bf L}$_{o,o}$ for linear functionals seeing $\mathcal{B}(X \times Y; \K)$ as $(X \pten Y)^*$,
	\item[(B)] the {\bf L}$_{o,o}$ for operators seeing $\mathcal{B}(X \times Y; \K)$ as $\mathcal{L}(X, Y^*)$, and, of course, 
	\item[(C)] the {\bf L}$_{o,o}$ for bilinear forms. 
\end{itemize}

We have the following relation between properties (A), (B), and (C): 

\begin{itemize}
\item {\it General implications.} Clearly, we have that (C) $\Rightarrow$ (B) by considering the associated bilinear for $B_T \in \mathcal{B}(X \times Y; \K)$ of a given operator $T \in \mathcal{L}(X, Y^*)$. Also, by our Theorem B (implication (iii) $\Rightarrow$ (i)) and noticing that, for this implication, we do not need any assumption on $X$ besides reflexivity (not even approximation property assumptions, see Remark \ref{remark:ap}), we also have that (C) $\Rightarrow$ (A). 

\vspace{0.2cm}

\item {\it Not true implications.} (B) does not imply (A) or (C) in general. Indeed, by \cite[Theorem 2.4.10]{alka}, for every $1 < p < \infty$, we have that $\mathcal{L}(c_0, \ell_p) = \mathcal{K}(c_0, \ell_p) = \mathcal{L}(\ell_{p'}, \ell_1)$, where $p'$ is the conjugate index of $p$. We have that $(\ell_{p'}, \ell_1)$ has the {\bf L}$_{o,o}$ for operators by Theorem A (since $(\ell_{p'}, \K)$ has the {\bf L}$_{o,o}$ for linear functionals) but neither $(\ell_{p'} \times c_0; \K)$ nor $(\ell_{p'} \pten c_0; \K)$ can have the {\bf L}$_{o,o}$ for bilinear forms and for linear functionals, respectively, since $c_0$ is not reflexive.

\vspace{0.2cm} 

\item {\it With extra assumptions implications.} Assume that either $X$ or $Y$ has the AP. In this case, implication (A) $\Rightarrow$ (B) holds. Indeed, if $(X \pten Y; \K)$ has the {\bf L}$_{o,o}$ for linear functionals, then $X \pten Y$ must be reflexive and, by the assumption that $X$ or $Y$ has the AP, every operator from $X$ into $Y^*$ is compact and by Theorem B (implication (i) $\Rightarrow$ (ii)), the pair $(X, \K)$ has the {\bf L}$_{o,o}$ for linear functionals. By Theorem A, the pair $(X, Y^*)$ has the {\bf L}$_{o,o}$ for operators. Finally, if $X$ or $Y$ has the AP and $X$ or $Y$ is stricly convex, then (A)$\Rightarrow$(C)
\end{itemize}

\noindent 
\textbf{Acknowledgements:}  S. Dantas was supported by Spanish AEI Project PID2019 - 106529GB - I00 / AEI / 10.13039/501100011033 and also by PGC2018 - 093794 - B - I00 (MCIU/AEI/FEDER, UE). A. Rueda Zoca was supported by Juan de la Cierva-Formaci\'on fellowship FJC2019-039973, by MTM2017-86182-P (Government of Spain, AEI/ FEDER, EU), by MICINN (Spain) Grant PGC2018-093794-B-I00 (MCIU, AEI, FEDER, UE), by Fundaci\'on S\'eneca, ACyT Regi\'on de Murcia grant 20797/PI/18, by Junta de Andaluc\'ia Grant A-FQM-484-UGR18 and by Junta de Andaluc\'ia Grant FQM-0185.

\end{document}